\documentclass[12pt]{amsart}
\usepackage{amscd,amssymb,graphics}
\usepackage{hyperref}

\usepackage{amsfonts}
\usepackage{amsmath}
\usepackage{amsxtra}
\usepackage{latexsym}
\usepackage[mathcal]{eucal}

\input xy
\xyoption{all}
\usepackage{epsfig}

\newtheorem{thm}{Theorem}[section]

\newtheorem{cor}[thm]{Corollary}
\newtheorem{lem}[thm]{Lemma}
\newtheorem{prop}[thm]{Proposition}

\theoremstyle{definition}
\newtheorem{defin}[thm]{Definition}

\newtheorem{ex}[thm]{Example}

\theoremstyle{remark}
\newtheorem{remark}[thm]{Remark}
\newtheorem{remarks}[thm]{Remarks}

\numberwithin{equation}{section}

\newcommand{\delete}[1]{} 
\newcommand{\nt}{\noindent}

\def\eps{{\varepsilon}}
\newcommand{\sk}{\vskip 0.2cm}

\newcommand{\ben}{\begin{enumerate}}

\newcommand{\een}{\end{enumerate}}
\newcommand{\bit}{\begin{itemize}}

\newcommand{\eit}{\end{itemize}}

\def\R {{\mathbb R}}

\def\Z {{\mathbb Z}}

\def\P {{\mathbb P}}

\newcommand{\id}{{\rm{id}}}

\def\QED{\nobreak\quad\ifmmode\roman{Q.E.D.}\else{\rm Q.E.D.}\fi}

\def\a{\alpha}

\newcommand{\la}{\lambda}

\newcommand{\ep}{\epsilon}

\begin{document}

\title[]
{On fixed point theorems and nonsensitivity}

\author[]{Eli Glasner}
\address{Department of Mathematics,
Tel-Aviv University, Ramat Aviv, Israel}
\email{glasner@math.tau.ac.il}
\urladdr{http://www.math.tau.ac.il/$^\sim$glasner}

\author[]{Michael Megrelishvili}
\address{Department of Mathematics,
Bar-Ilan University, 52900 Ramat-Gan, Israel}
\email{megereli@math.biu.ac.il}
\urladdr{http://www.math.biu.ac.il/$^\sim$megereli}

\date{September 28, 2010}

\begin{abstract}
Sensitivity is a prominent aspect of chaotic behavior of a
dynamical system. We study the relevance of
nonsensitivity to fixed point theory in affine dynamical systems.
We prove a fixed point theorem which extends Ryll-Nardzewski's theorem
and some of its generalizations.
Using the theory of hereditarily nonsensitive dynamical systems we
establish left amenability of $Asp(G)$,
the algebra of Asplund functions on a topological group $G$ (which contains
the algebra $WAP(G)$ of weakly almost periodic functions).
We note that, in contrast to $WAP(G)$
where the invariant mean is unique,
for some groups (including the integers)
there are uncountably many invariant means on $Asp(G)$.
Finally we observe that dynamical systems in the larger class of tame
$G$-systems need not admit an invariant probability measure.
\end{abstract}


\thanks{This research was partially supported by Grant No 2006119
from the United States-Israel Binational Science Foundation (BSF)}


\maketitle

\setcounter{tocdepth}{1}


\section*{Introduction}

Let $S$ be a semigroup, $X$ a topological space, and $S \times X \to X$ a semigroup action of $S$ on $X$ such that the
translations $\la_s: X \to X,\ s \in S$, written usually as
$\la_s(x) = sx$, are continuous maps.
We will say that the pair $(S,X)$ is a {\em dynamical system},
or that $X$ is an {\em $S$-system}.
If in addition $X=Q$ is a convex and compact subset of a locally
convex vector space and each $\la_s: Q \to Q$ is an affine map,
then the $S$-system $(S,Q)$ is called an {\em affine dynamical system}.
We use the symbol $G$ instead of $S$ when dealing with group actions, and
we require in this case that the group identity acts as the identity map. 
The topological and locally convex vector spaces (over the reals) in this paper are assumed to be Hausdorff.

 Let $\xi$ be a uniform structure on an $S$-system $X$. We say that the action of $S$ on $X$ 
(or, just $X$, or $S$, where the action is understood) is $\xi$-\emph{distal} 
if every pair $x, y$ of distinct points in $X$ is $\xi$-distal, i.e.,
there exists an entourage $\eps \in \xi$ such that
$$
(sx,sy) \notin \eps \ \ \forall s \in S.
$$ 
We recall the following well known fixed point theorem
of Ryll-Nardzewski \cite{Ryll}.

\begin{thm} \label{Ryll-Na}
\emph{(Ryll-Nardzewski)}
Let $V$ be a locally convex vector space equipped
with its uniform structure $\xi$.
Let $Q$ be an affine compact $S$-system such that
\ben
\item
$Q$ is a weakly compact subset in $V$.
\item
$S$ is $\xi$-distal on $Q$.
\een
Then $Q$ contains a fixed point.
\end{thm}

In the special case where $Q$ is compact already in the $\xi$-topology,
we get an equivalent version of Hahn's fixed point theorem \cite{Hahn}.
There are several geometric proofs of Theorem \ref{Ryll-Na}, see
Namioka and Asplund \cite{N-A}, Namioka \cite{Na-Isr, Na72, Na-affine},
Glasner \cite{Gl-comp, Gl-book}, Veech \cite{Veech},
and Hansel-Troallic \cite{Ha-Tr}.
The subject is treated in several books,
see for example \cite{Gl-book}, Berglund-Junghenn-Milnes \cite{BJM},
and Granas-Dugundji \cite{Gra-Dug}.

A crucial step in these proofs is the lifting of distality
on $Q$ from $\xi$ to the original compact topology.

In Section \ref{Sec-fp} we present a short proof of a fixed point theorem
(Theorem \ref{general-fpt}) which
covers several known generalizations of Theorem \ref{Ryll-Na}
(see Corollary \ref{old-res}).
Moreover, we apply Theorem \ref{general-fpt} in some cases where
Ryll-Nardzewski's theorem, or its known generalizations,
do not seem to work.
See for example Corollary \ref{c:NP}, where we apply our results to
weak-star compact affine dynamical systems in a large class of locally convex spaces.

The main tools of the present paper are the concepts of nonsensitivity and fragmentability.
The latter originally comes from Banach space theory and has several applications in Topology and recently also in Topological Dynamics.
Fragmentability (or the weaker concept of nonsensitivity) allows us in
Lemma \ref{lifting} to simplify and
strengthen the methods of Veech and Hansel-Troallic for lifting the distality property.
As in the proofs of Namioka \cite{Na72} and Veech \cite{Veech}, the strategy
is to reduce the problem at hand to the situation where the existence of an invariant measure follows from the following fundamental theorem of
Furstenberg \cite{Furst63}.

\begin{thm} \label{Furst}
\emph{(Furstenberg)}
Every distal compact dynamical system admits an invariant probability measure.
\end{thm}

This result was proved by Furstenberg
for metric dynamical systems using his structure theorem for minimal distal metric
$G$-systems (where $G$ is a group).
The latter was extended to general compact $G$-systems by Ellis \cite{Ellis-P},
and consequently Theorem \ref{Furst} is valid for
nonmetrizable $G$-systems as well.
Now from Ellis' theory it follows that the enveloping semigroup
of a distal semigroup action is actually a group and this fact makes it
possible
to
extend Furstenberg's theorem to distal semigroup actions.
See e.g. Namioka's work \cite{Na72},
where a proof of Theorem \ref{Furst} is obtained
as a fixed point theorem.

In Section \ref{Sec-HNS} we discuss the role of hereditarily nonsensitive
dynamical systems and the existence of invariant probability measures.
As was shown in \cite{GM1}, a metric compact $G$-system is
hereditarily nonsensitive (HNS) iff it can be linearly represented on
a separable Asplund Banach space $V$.
It follows that  the algebra $Asp(G)$, of functions on a topological group $G$
which come from
HNS (jointly continuous
\footnote{In this context the topology on $G$ becomes relevant}) $G$-systems,
coincides with the collection of functions
which appear as matrix coefficients of continuous co-representations of
$G$ on Asplund Banach spaces.
Replacing Asplund by reflexive, gives the characterization (see \cite{Me-nz}) of the algebra $WAP(G)$
of weakly almost periodic functions.
Since every reflexive space is Asplund
we have $WAP(G) \subset Asp(G)$.
Refer to  \cite{Me-nz, GM1, GMU, GM-rose} and the review
article \cite{Gl-env} for more details about HNS, $Asp(G)$ and representations
of dynamical systems on Asplund and other Banach spaces.

From the theory of HNS dynamical systems, as developed in \cite{GM1}, we deduce the
existence of a left invariant mean on $Asp(G)$ (Proposition \ref{Asp-amen}).
We note however that, in contrast to the uniqueness of the invariant
mean on $WAP(G)$, there are, in general, many different invariant means
on $Asp(G)$.

In Section \ref{Sec-tame} we observe that the
still larger
algebra $Tame(G)$, of tame functions on $G$, is not, in general, amenable.
Equivalently, tame dynamical systems need not admit an invariant probability measure.
This is a bit surprising as the class of tame dynamical systems, although it contains
many sensitive dynamical systems, can still be considered as
non-chaotic in the sense that its members lie on the ``tame" side of the
Bourgain-Fremlin-Talagrand dichotomy (see \cite{GM1,Gl-tame,Gl-env, GM-rose}).

We are grateful to I. Namioka for sending us his manuscript \cite{Na10}.

\section{A generalization of Ryll-Nardzewski's fixed point theorem}\label{Sec-fp}

\subsection{Sensitivity and fragmentability}
\label{sens}

Let $(X,\tau)$ be a topological space and $(Y,\xi)$ a uniform space.
We say that $X$ is {\em $(\tau, \xi)$-fragmented\/} by a
(typically not continuous)
function $\a: X \to Y$ if for every
nonempty subset $A$ of $X$ and every $\eps \in \xi$, there exists
an open subset $O$ of $X$ such that $O \cap A$ is nonempty and
$\a(O \cap A)$ is $\eps$-small in $Y$. Note that it is
enough to check the condition above for closed subsets $A
\subset X$.

This definition of fragmentability is a slight generalization of the original one which is due to Jayne and Rogers \cite{JR}.
It appears implicitly in a work of Namioka and Phelps \cite{NP} which deals
with a characterization of Asplund Banach spaces $V$ in terms of (weak$^*$,norm)-fragmentability (Lemma \ref{fr-cases}.1),
whence the name
\emph{Namioka-Phelps spaces}
in the locally convex version of Asplund spaces given in Definition \ref{NP} below.
See \cite{Na-RN, me-fr, Me-nz, GM1, GM-rose} for more details.

Let again $\a: X \to Y$ be a (typically not continuous) map of a
topological
space $(X,\tau)$ into a uniform space $(Y,\xi)$.
We say that $X$ is \emph{$(\tau,\xi)$-nonsensitive} (with respect to $\a$),
or simply \emph{$\xi$-nonsensitive}, when $\tau$ is understood,
if for every $\eps \in \xi$ there exists a non-void open subset $O$ in $X$
such that $\a(O)$ is $\eps$-small.
Thus $X$ is $(\tau,\xi)$-fragmented iff every non-void (closed) subspace $A$ of $X$ is $\xi$-nonsensitive with respect to the
restricted map $\a_A: A \to X$.


Now let $X$ be a compact $S$-system endowed with its unique compatible
uniform structure $\mu$.
The $S$-system $(X,\mu)$ is {\em nonsensitive\/}, NS
for short, if for every $\ep \in \mu$ there exists an
open nonempty subset $O$ of $X$ such that $sO$ is $\eps$-small in
$(X, \mu)$ for all $s \in S$.
We say that an $S$-system $X$ is
\emph{hereditarily nonsensitive} (HNS) if every closed
$S$-subsystem
of $X$ is nonsensitive. Note that for a minimal $S$-system
nonsensitivity is the same as hereditary nonsensitivity.

If we let $\mu_S$ be the uniform structure on $X$ generated by the
entourages of the form $\epsilon_S =
\{(x,x') \in X \times X: (sx,sx') \in \epsilon,\ \forall s \in S\}$
for $\epsilon \in \mu$,
then hereditary nonsensitivity is equivalent to the requirement that
the identity map $\id : (X,\mu) \to (X,\mu_S)$ be fragmented.
For more details about (non)sensitivity of dynamical systems refer e.g. to \cite{AAB96, GW1, GM1}.


\vspace{0.4cm}

As was shown by
Namioka \cite{Na-RN}, every weakly compact subset $(X,\tau)$
in a Banach space $V$ is $(\tau, norm)$-fragmented
(with respect to the map $id: (X,\tau) \to (X,norm)$).
We need the following generalization.

\begin{lem} \label{l:fr}
\cite[Prop. 3.5]{me-fr}
Every weakly compact subset $(X,\tau)$ in a locally convex space $V$ is
$(\tau, \xi)$-fragmented,
where $\xi$ is the natural uniform structure of $V$.
\end{lem}
\begin{proof}
For completeness we give a sketch of the proof.
The topology of a locally convex space $V$
coincides (see \cite[Ch. IV, 1.5, Cor. 4]{Sch}) with the topology of uniform
convergence on equicontinuous subsets of $V^*$.
By the Alaouglu-Bourbaki theorem every equicontinuous subset of $V^*$
is weak$^*$
precompact, where by the \emph{weak$^*$ topology} we mean the
usual $\sigma(V^*,V)$ topology on the dual $V^*$. Therefore, the collection
of subsets
$$
[K,\eps]=\{(v_1,v_2) \in V \times V| \ \ |f(v_1)-f(v_2)|
< \eps \ \ \forall f \in K\},
$$
where $K$ is a weak$^*$ compact
equicontinuous
subset in $V^*$ and $\eps >0$, forms a base for the uniform structure $\xi$
on $V$. In order to show that $X$ is $(\tau,\xi)$-fragmented we
have to check that for every closed nonempty subset $A$ of $X$ and
every $[K,\eps]$, there exists a $\tau$-open subset $O$
of $X$ such that $O \cap A$ is nonempty and $[K,\eps]$-small.
Since $(A, \tau)$ is weakly compact in $V$, the evaluation map
$\pi: A \times K \to \R$ is separately continuous. By Namioka's
joint continuity theorem,
\cite{Na-jct} Theorem 1.2,
there exists a point $a_0$ of $A$ such
that $\pi$ is jointly continuous at every point $(a_0,y)$, where $y \in
K$. Since $K$ is compact one may choose a $\tau$-open subset $O$
of $X$ containing $a_0$ such that $|f(v_1)-f(v_2)| < \eps$ for
every $f \in K$ and $v_1, v_2 \in O \cap A$.
\end{proof}

The following lifting lemma strengthens a result of Hansel and
Troallic \cite{Ha-Tr} which in turn was inspired by a technique developed by
Veech \cite{Veech}.

\begin{lem} \label{lifting}
Let $X$ be a compact minimal $S$-system with its unique compatible
uniform structure $\mu$.
Assume that $X$ is $\xi$-nonsensitive (e.g., $\xi$-fragmented) with
respect to an $S$-map $\a: X \to M$ into a uniform
space $(M,\xi)$, where the semigroup action of $S$ on $M$ is $\xi$-distal.
Then every pair $(x,y)$ in $X$ with distinct images $\a(x) \neq \a(y)$
is $\mu$-distal.
In particular, if $\a$ is injective then the $S$-action on $(X,\mu)$ is distal.
\end{lem}

\begin{proof}
Consider a pair of points $x, y \in X$ with $\a(x) \neq \a(y)$.
Since $M$ is $\xi$-distal there exists an entourage $\eps \in \xi$ such that
$$
(s \a(x),s \a(y)) \notin \eps \ \ \forall s \in S.
$$

As $X$ is $\xi$-nonsensitive, there exists a \emph{non-void}
$\mu$-open subset $O \subset X$ such that $\a(O)$ is $\eps$-small.
By minimality of $X$
$$
X = \bigcup_{s \in S} s^{-1}O,
$$
where $s^{-1}O=\{x \in X: \ sx \in O \}.$
Set
$$
\gamma : = \bigcup_{s \in S} (s^{-1}O \times s^{-1}O) \subset X \times X.
$$
Then $\gamma \in \mu$
(every open neighborhood of the diagonal in $X \times X$ for
a compact Hausdorff space $X$ is an element of the
unique compatible uniform structure).
Since $\a$ is an $S$-map one easily gets
$$
(sx,sy) \notin \gamma \ \ \forall s \in S.
$$
\end{proof}

For later use
we list in Lemma \ref{fr-cases} some additional situations where
fragmentability appears.
First recall some necessary definitions.
A Banach space $V$ is called \emph{Asplund}
if the dual of every separable Banach subspace of $V$ is separable.
We say that a Banach space $V$ is \emph{Rosenthal} if it does not
contain an isomorphic copy of $l_1$
\cite{GM-rose}.
A uniform space $(X,\xi)$ is called
\emph{uniformly Lindel\"{o}f} \cite{me-fr} (or \emph{$\aleph_0$-precompact}
\cite{Isb})
if for every $\eps \in \mu$ there exists a countable
subset $A \subset X$ such that $A$ is $\eps$-dense in $X$.

\begin{lem} \label{fr-cases}
\ben
\item
\cite{Na-RN} A Banach space $V$ is Asplund iff every bounded subset of
the dual $V^*$ is $(weak^*,norm)$-fragmented.
\item
\cite{GM-rose} A Banach space
$V$ is Rosenthal iff every bounded subset of the dual $V^*$ is
$(weak^*,weak)$-fragmented.
\item
\cite{Na-RN} A topological space $(X,\tau)$ is {\em scattered\/} (i.e., every
nonempty subspace has an isolated point)
iff $X$ is $(\tau, \xi)$-fragmented for any uniform structure $\xi$ on the set $X$.
A compact space $X$ is scattered iff the Banach space $C(X)$ is Asplund.
\item
Let $(X, \tau)$ be a compact space and $\xi$ a uniform structure on the set $X$.
Assume that $(X, \xi)$ is uniformly Lindel\"{o}f (e.g., $\xi$-separable)
and that there exists
a base for the uniformity $\xi$ consisting of $\tau$-closed subsets of $X \times X$.
Then $X$ is $(\tau, \xi)$-fragmented.
\item
\cite[Prop. 6.7]{GM1} If $X$ is a Polish space and $\xi$ a metrizable
separable uniform structure
on $Y$
then $f: X \to Y$ is fragmented iff $f$ is a Baire 1 function.
\een
\end{lem}
\begin{proof}
(4) It is easy to check,
using Baire category theorem, that $X$ is $(\tau,\xi)$-fragmentable.
\end{proof}

\subsection{Fixed point theorems}

An \emph{$S$-affine compactification} 
of an $S$-system $X$ is a pair $(Q,\phi)$ where
$Q$ is a compact convex affine $S$-system,
%
and $\phi: X \to Q$ is a continuous $S$-map such that $\overline{co} \phi(X) =Q$.
See \cite{GM-OC} for a detailed exposition.

If $X$ is a compact $S$-system then the natural embedding
$\delta: X \to P(X)$ into the affine compact $S$-system $P(X)$ of
probability measures on $X$, defines an $S$-affine compactification
$(P(X),\delta)$.
Moreover this  $S$-affine compactification
is \emph{universal} in the sense that
for any other  $S$-affine compactification $(Q,\phi)$ of $X$ there
exists a uniquely defined continuous affine surjective
$S$-map $\textsl{b}: P(X) \to Q$, called the
\emph{barycenter map}, such that $\textsl{b} \circ \delta=\phi$.

\begin{defin} \label{d:afpt}
A
(not necessarily compact) $S$-system $X$ has the {\em affine fixed point (a.f.p.) property} if whenever $(Q,\phi)$ is an
$S$-affine
compactification of $X$, then the dynamical system $Q$ has a fixed point.
When $X$ is compact, in view of the remark above,
this is equivalent to saying that $X$ admits an $S$-invariant probability measure.
\end{defin}


\begin{thm} \label{afpt}
Let $(X,\tau)$ be a compact $S$-system and $(M,\xi)$
a uniform space equipped with a semigroup action of $S$.
Suppose
\ben
\item
There exist a compact subsystem (minimal subsystem)
$Y \subset X$ and an injective
$S$-map $\a:~Y \to M$ such that $Y$ is $(\tau,\xi)$-fragmented
(respectively, $(\tau,\xi)$-nonsensitive).
\item
The action of $S$ on $\a(Y)$ is $\xi$-distal.
\een
Then the $S$-system $X$ has the affine fixed point property.
\end{thm}

\begin{proof}
Let $(Q,\phi)$ be an $S$-affine compactification of $X$.
Let $Y \subset X$ be a $\tau$-compact subsystem which satisfies
the conditions (1) and (2).
Since the $s$-translations $\lambda_s: Q \to Q$ are
continuous,
the closed convex hull
$Q_0=\overline{co} (Y)$
is $S$-invariant.

Fragmentability is a hereditary property, hence in any case we may assume
that $Y$ is minimal and $(\tau,\xi)$-nonsensitive.
Applying Lemma \ref{lifting} to the map
$\alpha: (Y, \tau) \to (\a(Y), \xi)$,
we see that the $S$-system $Y$ is $\tau$-distal.
By Furstenberg's theorem \ref{Furst}
the distal dynamical system $(S,Y,\tau)$ admits an
invariant probability measure.
Therefore, the compact $S$-system $P(Y)$ has a fixed point.
Since $Q_0$ is an $S$-factor of $P(Y)$
via the barycenter map
$b: P(Y) \to Q_0$,
we conclude that $Q_0$,
and hence also $Q$, admit a fixed point.
\end{proof}

Lemma \ref{l:fr} shows that the following result
is indeed a generalization of Ryll-Nardzewski's
fixed point theorem.

\begin{thm} \label{general-fpt}
Let $\tau_1$ and $\tau_2$ be two locally convex topologies on a vector
space $V$ with their uniform structures $\xi_1$ and $\xi_2$ respectively.
Assume that $S \times Q \to Q$ is a semigroup action such that
$Q$ is an affine $\tau_1$-compact $S$-system.
Let $X$ be an  $S$-invariant $\tau_1$-closed subset of $Q$ such that:
\ben
\item
$X$ is either $(\tau_1,\xi_2)$-fragmented, or $X$ is minimal and
$(\tau_1,\xi_2)$-sensitive.
\item
the $S$-action is $\xi_2$-distal on $X$.
\een
Then $Q$ contains a fixed point.
\end{thm}
\begin{proof}
Applying Theorem \ref{afpt} to the map $id: (X,\tau_1) \to (Q, \xi_2)$
it follows that $X$ has the a.f.p. property.
Hence the compact affine $S$-system $Q_0:=\overline{co} X$ has a fixed point,
which is also a fixed point of $Q$.
\end{proof}

\begin{cor} \label{old-res}
Theorem \ref{general-fpt}
includes in particular the following results:
\ben
\item
Ryll-Nardzewski's theorem \ref{Ryll-Na}.
\item
Furstenberg's theorem \ref{Furst} and its generalized version of
Namioka \cite[Theorem 4.1]{Na72}.
\item
Veech's theorem concerning weakly compact subsets in Banach spaces \cite[Cor. 2.5]{Veech}
and its locally convex version of Namioka
\emph{(see \cite[p. 361]{Veech} and \cite[Thm 5.1]{Na10})}.
\item
Namioka-Phelps' theorem \cite[p. 745]{NP} about weak-star compact convex subsets
in the dual $V^*$ of an Asplund Banach space $V$ \emph{(see also Proposition \ref{NP}
and Remark \ref{NP-extent} below)}.
\item
Assume in the hypotheses of Theorem \ref{general-fpt} that condition
(1) is replaced by
\sk
\bit
\item
[($\star$)] \ \   $X \subset V$ is
$\xi$-separable (or, more generally, uniformly Lindel\"{o}f)
and there exists
a base for the uniformity $\xi$ consisting of $\tau$-closed subsets of $X \times X$.
\eit
\sk
Then $Q$ contains a fixed point.
\een
\end{cor}
\begin{proof}
(1) Apply Theorem \ref{general-fpt} (with $X=Q$) and
Lemma \ref{l:fr}.

(2) Let $V$ be the locally convex space $(C(X)^*, w^*)$,
with its weak-star topology. Let $\xi$ be the corresponding
uniform structure and let $Q=P(X)$. Thus, in this case
$\tau_1=\tau_2=w^*$ and $\xi_1=\xi_2=\xi$
coincide on $X \subset C(X)^*$.
Hence, in particular, $X$ is $(\tau_1,\xi_2)$-fragmented and
$S$ is $\xi_2$-distal on $X$.
(Of course this is not a new proof of Furstenberg's theorem, as our
proof of Theorem \ref{general-fpt} relies on it.
This is merely the claim that conversely, Furstenberg's theorem also
follows from Theorem \ref{general-fpt}.)

(3) We need, as in (1), to apply Lemma \ref{l:fr} (but now
$X$ is not necessarily all of $Q$).

(4) Recall that by Lemma \ref{fr-cases}.1  weak$^*$ compact subsets
in the dual of an Asplund space $V$ are (weak$^*$, norm)-fragmented.

(5) Apply Lemma \ref{fr-cases}.4 and Theorem \ref{general-fpt}.
\end{proof}

\begin{remark} \label{rem}
(1)\
In cases where
the distality can be extended to (or is assumed on) all of
$Q$ the existence of a fixed point can be achieved without the
use of Furstenberg's theorem \ref{Furst}, either by Hahn's fixed point
theorem or via Glasner's results using the concept of
\emph{strong proximality}
\cite{Gl-comp, Gl-book} (see also Example \ref{ex} below).

(2) Namioka and Phelps noticed \cite[p. 745]{NP}
that Ryll-Nardzewski's theorem is not generally true in dual
spaces $V^*$
when the weak topology is replaced by the weak$^*$ topology.
Thus the assumption that $V$ is Asplund in Corollary \ref{old-res}.4
is essential.

(3) \
Case (5) of Corollary \ref{old-res}
strengthens a result of Namioka \cite[Theorem 3.7]{Na-Isr}
and covers the results of Hansel-Troallic \cite{Ha-Tr}.
The latter, and also \cite[p.174]{Gra-Dug}, use the standard reduction to
the case where $S$ is
countable and $V$ is (weakly) separable.
\end{remark}

\subsection{The dual system fixed point property and Namioka-Phelps spaces}

As mentioned in Lemma \ref{fr-cases}.1, a Banach space $V$ is Asplund
iff every bounded subset of its dual is (weak$^*$, norm)-fragmented.
This fact together with Theorem \ref{general-fpt} and Remark \ref{rem}.2
suggest Definition \ref{d-fpp} below.
First, a few words of explanation.
For a locally convex space $V$,
the standard uniform structure $\xi^*$ of the dual $V^*$ is the uniform structure of
uniform convergence on the family of all bounded subsets of $V$.
By the Alaoglu-Bourbaki theorem every equicontinuous subset $Q$ of
$V^*$ is relatively weak$^*$ compact. Conversely,
if $V$ is a barreled space (or, if $V$ is Baire as a
topological space) then it follows from the generalized
Banach-Steinhaus theorem (see \cite[Ch. III, \S 4.2]{Sch})
that every weak$^*$ compact subset of $V^*$ is equicontinuous.
Clearly, if $V$ is a normed space then the equicontinuous subsets of the dual $V^*$ are exactly the norm bounded subsets.

\begin{defin} \label{d-fpp}
(a)\
We say that
a Banach space $V$ has the \emph{dual system fixed point property}
if for
every semigroup $S$, every convex weak$^*$ compact
norm-distal affine $S$-system $Q \subset V^*$ has a fixed point.

(b)\
More generally,
a locally convex space $V$ has the \emph{dual system fixed point property}
if whenever $Q \subset V^*$ is a weak$^*$ compact convex affine
$S$-system such that
(1) $Q$ as a subset of $V^*$
is equicontinuous and (2)
$S$ is $\xi^*$-distal on $Q$,
then $Q$ has a fixed point.
(Note that if $V$ is
barrelled then we may drop the assumption
(1)).
\end{defin}

Definition \ref{d-fpp}
and Theorem \ref{general-fpt} lead to the study of locally convex vector spaces $V$
such that every ($w^*$-compact) equicontinuous subset $K$ in $V^*$ is
(weak$^*$,$\xi^*$)-fragmented. This is a locally convex version of Asplund
Banach spaces.
In fact, this definition was already introduced in \cite{me-fr}, where
it was motivated by problems concerning continuity of dual actions.
A typical result of \cite{me-fr} asserts that if $V$ is
an Asplund Banach space then for every
continuous linear action of a topological group $G$ on $V$ the
corresponding dual action of $G$ on $V^*$ is continuous.

\begin{defin} \label{NP} \cite{me-fr}
A  locally convex space $V$ is called a \emph{Namioka-Phelps space},
(NP)-space for short, if every equicontinuous
subset $K$ in $V^*$ is (weak$^*$,$\xi^*$)-fragmented.
\end{defin}

Now by Theorem \ref{general-fpt} we get:

\begin{cor} \label{c:NP}
Every (NP) locally convex space has the dual system fixed point property.
\end{cor}

\begin{remark} \label{NP-extent}
Recall that the class (NP) is quite large and contains:
\ben
\item
Asplund (hence, also reflexive) Banach spaces.
\item
Frechet differentiable
spaces.
\item
Semireflexive locally convex spaces.
\item
Quasi-Montel (in particular, nuclear) spaces.
\item
Locally convex spaces $V$ having uniformly Lindel\"{o}f $V^*$
(equivalently,
$V^*$ is a subspace in a product of separable locally convex spaces).
\een
The class (NP) is closed under subspaces,
continuous bound covering linear operators, products
and locally convex direct sums. See \cite{me-fr} for more details.
\end{remark}

\section{Hereditary nonsensitivity and invariant measures}
\label{Sec-HNS}

\subsection{Affine dynamical systems admitting a fixed point}
In Theorem \ref{general-fpt} and its prototype \ref{Ryll-Na}
an additional ``external" condition is imposed on the affine dynamical system $Q$.
The following proposition characterizes, in the case of a group
action, those affine dynamical systems which admit a fixed point.

 \begin{prop} \label{prop-fp}
 Let $Q$ be an affine compact $G$-system, where $G$ is a
 group.
 Then the following conditions are equivalent:
 \ben
 \item $Q$ contains a fixed point.
 \item $Q$ contains a scattered
compact subsystem.
 \item $Q$ contains a HNS compact subsystem.
 \item $Q$ contains an equicontinuous compact subsystem.
 \item $Q$ contains a distal compact subsystem.
 \item
 There exist a compact subsystem (minimal subsystem) $Y \subset X$,
 a uniform space $(M,\xi)$ with a $\xi$-distal action of $G$ on $M$,
 and an injective
$G$-map $\a: Y \to M$ such that $Y$ is $(\tau,\xi)$-fragmented
(resp., $(\tau,\xi)$-nonsensitive).
 \item
 $Q$ contains a compact subsystem admitting an invariant probability measure.
\een
\end{prop}
\begin{proof} (1) $\Rightarrow$ (2) Is trivial.

(2) $\Rightarrow$ (3) Every scattered compact $G$-system $X$ is HNS.
In fact, observe that $X$, being scattered, is $(\tau, \xi)$-fragmented (Lemma \ref{fr-cases}.3) for any
uniform structure $\xi$ on the set $X$.
Now see the definition of HNS as in Subsection \ref{sens}.

\nt \emph{A second proof:}
As $C(X)$ (by Lemma \ref{fr-cases}.3) is Asplund,
the regular dynamical system representation of
$G$ on $C(X)$ ensures that $X$ is Asplund representable.
This implies that $X$ is HNS by \cite[Theorem 9.9]{GM1}.

(3) $\Rightarrow$ (4)
Assume that $Q$ contains a HNS compact subsystem $X$.
Then any minimal compact $G$-subsystem $Y$ of $X$ is equicontinuous by
\cite[Lemma 9.2.3]{GM1}.

(4) $\Rightarrow$ (5) This is well known and easy to see for
\emph{group} actions on compact spaces
(it is not, in general, true for semigroup actions).

(5) $\Rightarrow$ (6) Consider the identity map $\a: X \to M=X$ and
let $\xi$ be the compatible uniform structure on $X$.

(6) $\Rightarrow$ (7) Follows from Theorem \ref{afpt} and
Definition \ref{d:afpt}.

(7) $\Rightarrow$ (1) As in the proof of Theorem \ref{afpt}
use the barycenter map.
\end{proof}

\begin{prop} \label{inv-m}
Every HNS compact $G$-system $X$ admits an invariant probability measure.
\end{prop}
\begin{proof}
The compact affine $G$-system $P(X)$ contains $X$
as a subsystem which is HNS. Thus, Proposition \ref{prop-fp} applies.
\end{proof}


\subsection{HNS dynamical systems, Asplund functions and amenability of $Asp(G)$}
\label{s:Asp}

In this subsection $G$ will denote a topological group and a ``$G$-system"
will mean a dynamical system with a jointly continuous action. 
In fact the results remain true for semitopological groups 
\footnote{A semitopological group is a group endowed with a topology
with respect to which multiplication is separately continuous.}
but for simplicity we consider only the case of topological groups. 


Recall that a (continuous, bounded ) real valued function $f: G \to \R$  
is an {\em Asplund function}, if there is a HNS compact $G$-system $X$,
a continuous function $F: X \to \R$, and a point $x_0 \in X$ such that
$F(gx_0)=f(g)$, for every $g \in G$. Every $f \in Asp(G)$ is right and left uniformly continuous. 
The collection $Asp(G)$ of Asplund functions is a uniformly closed $G$-invariant subalgebra of
$l_{\infty}(G)$ and $Asp(G)$ contains the algebra $WAP(G)$ of weakly
almost periodic functions on $G$. Refer to \cite{Me-nz, GM1} for more details.

A left translation $G$-invariant normed subspace $F \subset l_{\infty}(G)$
is said to be \emph{left amenable}
(see for example \cite{Gr} or  \cite{BJM})
if the affine compact $G$-system $Q=M(F)$ of means on $F$ has a fixed point,
a \emph{left invariant mean}. It is a classical result of Ryll-Nardzewski
\cite{Ryll}, that $WAP(G)$ is left amenable
\footnote{Note that $WAP(G)$, in addition, is also right amenable \cite{Ryll}.}.
We extend this result to $Asp(G)$.

\begin{prop} \label{Asp-amen}
$Asp(G)$ is left amenable.
\end{prop}
\begin{proof}
Denote by
$X:=|Asp(G)|$,
the Gelfand space of the
algebra $Asp(G)$.
By \cite[Theorem 9.9]{GM1} the dynamical system $X$ is HNS.
The Gelfand space $X$ can be identified with the space of
multiplicative means on the algebra
$V:=Asp(G)$. Thus $X$ is
embedded as a $G$-subsystem in the compact affine $G$-system $Q:=M(V)$ of means on $V$.

Let $Y$ be a minimal $G$-subsystem of $X$.
Then the
$G$-system
$Y$ is HNS as well.
Furthermore, $Y$ is equicontinuous by \cite[Lemma 9.2.3]{GM1}.
Thus $Q$ contains an equicontinuous compact $G$-subsystem $Y$ and
Proposition \ref{prop-fp} implies that $Q$ has a fixed point.
\end{proof}

\begin{cor} \label{Ryll-wap}
\emph{(Ryll-Nardzewski \cite{Ryll})}
$WAP(G)$ is left amenable.
\end{cor}

\begin{remarks} \label{remarks}
(1) \ Examples constructed in \cite{GW2} (together with Theorem 11.1 of
\cite{GM1})
show that a
point transitive HNS $\Z$-dynamical system can contain uncountably many
minimal subsets (unlike the situation in a point-transitive WAP-dynamical system
where there is always a unique minimal set). As a $\Z$-dynamical system, each
of these minimal sets supports an invariant measure, and since
our dynamical systems are factors of the universal HNS dynamical system $|Asp(\Z)|$,
it follows that the latter has uncountably many distinct invariant
measures. As there is a one-to-one correspondence between
invariant probability measures on $|Asp(G)|$ and invariant means on
the algebra $Asp(G)$ we conclude that, unlike $WAP(\Z)$ where
the invariant mean is unique, the algebra $Asp(\Z)$ admits uncountably
many invariant means.

(2)\
The group $G$ in Proposition \ref{Asp-amen} and Corollary \ref{Ryll-wap}
cannot be replaced, in general, by semigroups. Indeed recall \cite[p.147]{BJM}
that even for finite semigroups the algebra $AP(S)$ of the
almost periodic functions need not be left (right) amenable.

\end{remarks}


\section{Concerning tame dynamical systems}
\label{Sec-tame}

As we have already mentioned, a compact $G$-system $X$ is HNS iff it admits
sufficiently many representations on Asplund Banach spaces.
In a recent work \cite{GM-rose} we have shown that an analogous statement
holds for the family of tame dynamical systems and the larger class of Rosenthal Banach spaces.
A (not necessarily metrizable) compact $G$-system $X$ is said to be
\emph{tame} if for every element $p \in E(X)$ of the enveloping semigroup
$E(X)$ the function $p: X \to X$ is fragmented (equivalently,
Baire 1, for metrizable $X$;
see Lemma \ref{fr-cases}.5).

The algebra $Tame(G)$ of tame functions coincides with the collection of functions
which appear as matrix coefficients of continuous co-representations of
$G$ on Rosenthal Banach spaces.

One may ask if Propositions \ref{prop-fp}, \ref{inv-m} and  \ref{Asp-amen}
can be extended from HNS to tame dynamical systems.
The following counterexample shows that in general this is not the case.

\begin{ex} \label{ex}
There exists a tame minimal compact metric $G$-system $X$ such that
$P(X)$ does not have a fixed point (equivalently, $X$ does not have an
invariant probability measure).
\end{ex}
\begin{proof}
Take $X  =  \P^1$ to be the real projective line: all lines through the
origin in $\R^2$. Let $T$ be a parabolic M\"{o}bius transformation
(with a single fixed point), let $R = R_\alpha$ be a M\"{o}bius transformation
which corresponds to an irrational rotation of the circle.
Let $G =\langle T,R \rangle$ be the subgroup of $Homeo(X)$
generated by $T$ and $R$.
It is easy to see that the dynamical system $(G,X)$ is minimal.
Furthermore, every element $p$ of $E(X)$, the enveloping
semigroup of $(G,X)$, is a linear map. It can be shown that $p$ is
either in $GL(2,\R)$ or it maps all of $X \setminus \{x_0\}$ onto
$x_1$, where $x_0$ and $x_1$ are points in X. In particular every
element of the enveloping semigroup $E(X)$ is of Baire class 1.
This last fact implies that $X$ is tame. It is easily checked that $(G,X)$ is
\emph{strongly proximal}
in the sense of \cite{Gl-book} (that is, $P(X)$, as a $G$-system, is proximal),
and that $X$ is the unique minimal subset of $P(X)$.
Thus every fixed point of $P(X)$ is contained in $X$ and,
as $X$ is minimal, it follows that $X$ is trivial, a contradiction.
\end{proof}

\begin{cor}
There exists a
finitely generated
group $G$ for which the algebra $Tame(G)$ is not amenable.
\end{cor}
\begin{proof}
In Example \ref{ex} we described a metric tame minimal $G$-system $X$,
with $G$ a group generated by two elements, which does not
admit an invariant probability measure.
The Gelfand space $|Tame(G)|$ is the universal point-transitive
tame $G$-system; i.e., for every point-transitive tame
$G$-system $(G,Z)$ there is a surjective homomorphism $|Tame(G)| \to Z$.
In particular, we have such a homomorphism $\phi: |Tame(G)| \to X$.
Now, the amenability of $Tame(G)$ is equivalent to
the existence of a $G$-invariant mean on $Tame(G)$ which, in turn,
is equivalent to the existence of a $G$-invariant measure on $|Tame(G)|$.
However, if $\mu$ is such a measure then its image
$\nu:=\phi_*(\mu)$
is an invariant measure on $X$; but this contradicts Example \ref{ex}.
\end{proof}

Since every tame compact metric $G$-system admits a
faithful
representation on a Rosenthal Banach space \cite{GM-rose} it follows from Example \ref{ex}
that Rosenthal Banach spaces need not have the dual system fixed point property.

%


\bibliographystyle{amsplain}

\end{document}